\documentclass{amsart}
\usepackage{graphicx} 
\usepackage{amsmath}
\DeclareMathOperator{\arccosh}{arccosh}
\usepackage{amssymb}
\usepackage{bbm}
\usepackage{geometry}
\usepackage{enumitem}
\usepackage[english]{babel}
\usepackage{amsthm}
\usepackage{hyperref}
\usepackage{subcaption}
\usepackage{caption}
\usepackage{bm}
\usepackage{xcolor}
\usepackage{tikz}

\usepackage [english]{babel}
\usepackage [autostyle, english = american]{csquotes}
\MakeOuterQuote{"}

\title{Dynamics of Integer Zeroes of Homogeneous Quadratic Equations over $\mathbb{R}^3$}
\author{Alden Paige}
\date{July 2026}

\begin{document}
\address{A. Paige, Department of Mathematics, University of Manchester, M13 9PL}
\email{alden.paige@manchester.ac.uk}

\theoremstyle{plain}
\newtheorem{thm}{Theorem}
\newtheorem{lma}[thm]{Lemma}
\newtheorem{prop}[thm]{Proposition}
\newtheorem{cor}[thm]{Corollary}
\newtheorem{conj}[thm]{Conjecture}
\numberwithin{thm}{section}

\theoremstyle{definition}
\newtheorem{defn}[thm]{Definition}
\newtheorem{rem}[thm]{Remark}
\newtheorem{exmp}[thm]{Example}
\maketitle

\begin{abstract}
    Romik has presented a construction of a 1-dimensional dynamical system on the unit interval by developing an algorithm that returns the unique sequence of matrices associated with a positive primitive Pythagorean triple (in the sense of Barning), and projecting the map involved in this algorithm onto an appropriate 1-dimensional space via stereographic projection. Romik additionally computes the infinite, absolutely continuous invariant measure, and shows that the system is conservative and ergodic.

    Later, Cha et al.\ provided a method of calculating "Berggren trees", which are generalisations of the tree of positive primitive Pythagorean triples one may construct via Barning's theorem, except for different homogeneous quadratic equations in 3 variables. We present here a method of computing 1-dimensional dynamical systems induced from these Berggren trees following Romik's outline, and determine their absolutely continuous invariant measures by adapting the method of Keane.
\end{abstract}

\section{Introduction}
There are several ways in which Pythagorean triples can be described. One such way is the following.
\begin{prop}[Barning, \cite{pub:7151}]
Every positive primitive Pythagorean triple $(a,b,c)$ (i.e. Pythagorean triples that satisfy $a,b,c>0$ and $ \gcd (a,b)=1$) can be uniquely expressed in the form 
\[ \begin{bmatrix}
    a\\b\\c
\end{bmatrix} = M_{a_1} \cdot \ldots \cdot M_{a_n}\cdot  \bm{v},\]
where $\bm{v} \in \{ (3,4,5)^T, (4,3,5)^T\}$, and each $M_{a_i}$ is one of three matrices described below:
\begin{equation}\label{PythagMatrices--}
    M_1 = \begin{bmatrix}
        -1 & 2 & 2 \\
        -2 & 1 & 2 \\
        -2 & 2 & 3\\
    \end{bmatrix}, \quad 
    M_2 = \begin{bmatrix}
        1 & 2 & 2 \\
        2 & 1 & 2 \\
        2 & 2 & 3 \\
    \end{bmatrix}, \quad 
    M_3 = \begin{bmatrix}
        1 & -2 & 2 \\
        2 & -1 & 2 \\
        2 & -2 & 3\\
    \end{bmatrix}.
\end{equation}
\end{prop}
This result appears to have been determined by many people independently, though we choose to cite Barning \cite{pub:7151} since much of the literature also does. Romik \cite{Romik2004} uses this theorem as a unique encoding of the positive primitive Pythagorean triples by finite strings of three digits (with the additional information of which triple $\bm{v}$ is), and describes a dynamical system which serves as an algorithm by which one can determine the encoding of any given triple. In order to do this, Romik takes advantage of the bijection that exists between primitive Pythagorean triples and rational points of $S^1$, and defines the map on these rational points, which can be continuously extended to all points in the positive quadrant. By presenting an infinite measure for which the dynamical system is ergodic and conservative, Romik is able to use ergodic theory to determine some results about the behaviour of Pythagorean triples. 

In \cite{Cha2018}, Cha et al.\ investigate other quadratic forms on $\mathbb{R}^3$, and present a method for determining what matrices and "base triples" are required to completely describe all integer solutions to this quadratic form in some specified region  $S \subset \mathbb{R}^3$. They call these matrix-and-base-triple collections "Berggren trees" for the specified quadratic form and region $S$, which allows for a similar encoding of the integer solutions in $S$. One can ask whether, following the outline of Romik, we can develop an algorithm to determine the sequence associated with a given integer solution, and additionally whether the resulting map has an infinite conservative ergodic measure.

In \cite{Keane95}, Keane proposes a method by which Gauss may have deduced the invariant measure for the continued fraction map. This method involves considering any point $x \in [0,1)$ as a root of a quadratic polynomial, and defining a new map based on how the coefficients of the polynomial change when we apply the continued fraction map to the root $x$ of the polynomial. This is explained in more detail in Section \ref{fullbranchcase}. This idea has been utilised in some spirit in various articles in order to construct natural extensions, often referred to as planar extensions in these papers, for piecewise M\"{o}bius maps (see \cite{arnoux17,calta24,Nakada81,PANTI202250} for some examples). It is worth noting, however, that these articles explicitly do not allow for the cases where we have non-uniformly expanding piecewise M\"{o}bius transformations (as this causes the extensions to have non-finite measure), and also do not mention the application to analysing solutions to quadratic forms.

In this article, we present some examples of quadratic forms with their associated Berggren trees and matrices, and compute the associated dynamical system a la Romik. Additionally, we present a more immediate adaptation of Keane's method \cite{Keane95} to calculate their absolutely continuous infinite invariant measures, and subsequently prove that these systems are ergodic and conservative using standard techniques involving bounded distortion properties. The article is intended to motivate further study into analysing the solutions to quadratic forms on $\mathbb{R}^3$ via Romik's method, and additionally to give an alternative description of the adapted Keane method which highlights that the method is still applicable to find infinite invariant measures for non-uniformly contracting piecewise M\"{o}bius maps.

\subsection{Main Computed Examples}\label{Main Results}

This section presents (without proof) quadratic forms for which the corresponding dynamical systems have been constructed and the invariant measure has been calculated. To reiterate, for each quadratic form on $\mathbb{R}^3$ in question, we have 
\begin{itemize}[itemsep=-10pt]
    \item A matrix generating system for the solutions in some subset of $\mathbb{Z}^3$, \\
    \item An associated 1-dimensional map,\\
    \item An absolutely continuous, $\sigma$-finite, invariant, conservative, ergodic measure.
\end{itemize}
For details about how these are related, see Sections \ref{MatricesSection}, \ref{MapsSection}, and \ref{ExtensionSection} respectively . The results are now stated, the first included being Romik's map associated with Pythagorean triples for completeness and comparison:

\begin{thm}
    For the quadratic form $Q(\underline{x}) = x_1^2 + x_2^2 - x_3^2$, the matrix generating system for solutions in $\{ \underline{x}\in \mathbb{Z}^3 : x_1, x_2, x_3 > 0 \}$ consists of the 3 matrices in \eqref{PythagMatrices--} and base triples $(3,4,5)^T,$ $ (4,3,5)^T$. The corresponding 1-dimensional dynamical system up to conjugacy is 
    \begin{equation}
        T_1(t) = \begin{cases}
            \frac{t}{1-2t} & t \in \left(0, \frac{1}{3} \right), \\
            \frac{1-2t}{t} & t \in \left(\frac{1}{3}, \frac{1}{2} \right), \\
            \frac{2t-1}{t} & t \in \left( \frac{1}{2}, 1\right),\\
        \end{cases}
    \end{equation}
    which possesses the $\sigma$-finite, invariant, ergodic, conservative measure 
    \begin{equation}
        d\mu = \frac{1}{t(1-t)}\, dt.
    \end{equation}
\end{thm}
See \cite{Romik2004} for further results pertaining to this system. The following is the system relating to $120$-degree triangles, mentioned at the end of \cite{Romik2004} and in \cite{Cha2018}:

\begin{thm}
    For the quadratic form $Q(\underline{x}) = x_1^2 + x_1x_2 + x_2^2 - x_3^2$, the matrix generating system for the solutions in $\{ \underline{x} \in \mathbb{Z}^3 : x_1, x_2, x_3 > 0 \}$ consists of the five matrices 
    \begin{equation}
        \begin{gathered}
            N_1 = \begin{bmatrix}
                -3 & 1 & 4 \\
                -4 & -1 & 4 \\
                -6 & 0 & 7 \\
            \end{bmatrix}, \quad 
            N_2 = \begin{bmatrix}
                4 & 1 & 4 \\
                3 & -1 & 4 \\
                6 & 0 & 7 \\
            \end{bmatrix}, \quad 
            N_3 = \begin{bmatrix}
                4 & 3 & 4\\
                3 & 4 & 4 \\
                6 & 6 & 7 \\
            \end{bmatrix}, \\
            N_4 = \begin{bmatrix}
                -1 & 3 & 4 \\
                1 & 4 & 4 \\
                0 & 6 & 7 \\
            \end{bmatrix}, \quad 
            N_5 = \begin{bmatrix}
                -1 & -4 & 4  \\
                1 & -3 & 4 \\
                0 & -6 & 7
            \end{bmatrix},
        \end{gathered}
    \end{equation}
    along with the base triples $(5,3,7)^T,$ $(3,5,7)^T$, $(7,8,13)^T$, $(8,7,13)^T$. The corresponding 1-dimensional dynamical system (up to conjugacy) is 
    \begin{equation}
        T_2(t) = \begin{cases}
            \frac{t}{1-3t} & t \in \left(0, \frac{1}{4} \right), \\
            \frac{1-3t}{t} & t \in \left(\frac{1}{4}, \frac{1}{3} \right), \\
            \frac{3t-1}{1-2t} & t \in \left( \frac{1}{3}, \frac{2}{5} \right), \\
            \frac{1-2t}{3t-1} & t \in \left(\frac{2}{5}, \frac{1}{2} \right), \\
            \frac{2t-1}{t} & t \in \left(\frac{1}{2}, 1 \right),
        \end{cases}
    \end{equation}
    (with projection map $\pi(x,y) = \frac{y}{x+1}$, $\pi^{-1}(t) = \left( \frac{1-t^2}{1+t+t^2}, \frac{t^2 + 2t}{1+t+t^2}\right)$) and possesses the $\sigma$-finite, invariant, ergodic and conservative measure
    \begin{equation} d\mu = \frac{1}{t(1-t)}\, dt.\end{equation}
\end{thm}

The following two examples are those for which the explicit conjugation to a piecewise M\"{o}bius interval map is novel, where the associated matrix systems are explicitly computed in \cite{Cha2018}.

\begin{thm}
    For the quadratic form $Q(\underline{x}) = x_1^2 + x_2^2 - 2x_3^2$, the matrix generating system for the solutions in $\{x_1, x_2, x_3 \in \mathbb{Z}: x_2 > x_1 > 0, \ x_3>0 \}$ consists of the 3 matrices
    \begin{equation}
        N_1 = \begin{bmatrix}
            0 & 1 & 0 \\
            -3 & 0 & 4 \\
            -2 & 0 & 3 \\
        \end{bmatrix}, \quad 
        N_2 = \begin{bmatrix}
            0 & 1 & 0 \\
            3 & 0 & 4 \\
            2 & 0 & 3 \\
        \end{bmatrix}, \quad 
        N_3 = \begin{bmatrix}
            1 & 0 & 0 \\
            0 & 3 & 4 \\
            0 & 2 & 3 \\
        \end{bmatrix},
    \end{equation}
    along with the base triple $(1,7,5)^T$. The corresponding 1-dimensional dynamical system is given by 
    \begin{equation}
        T_3(t) = \begin{cases}
            \frac{t}{1-t} & t \in \left(0, \frac{\sqrt{2}}{1+ \sqrt{2}} \right),\\
            \frac{2-2t}{t} & t \in \left(\frac{\sqrt{2}}{1+ \sqrt{2}}, 1\right),\\
            \frac{2t-2}{2-t} & t \in (1, \sqrt{2} ),\\
        \end{cases}
    \end{equation}
    with the conjugacy map being $\pi(x,y) = \frac{y-x}{x+1}$, and possesses the $\sigma$-finite, invariant, ergodic and conservative measure
    \begin{equation}
        d\mu = \frac{1}{t} dt.
    \end{equation}
\end{thm}

\begin{thm}
    For the quadratic form $Q(\underline{x}) = x_1^2 + 2x_2^2 - x_3^2$, the matrix generating system for the solutions in $\{ x_1, x_2, x_3 \in \mathbb{Z} : x_1, x_2, x_3 > 0 \}$ consists of the 3 matrices
    \begin{equation}
        \begin{gathered}
            N_1 = \begin{bmatrix}
                0 & -2 & 1 \\
                1 & -1 & 1 \\
                1 & -2 & 2 \\
            \end{bmatrix}, \quad
            N_2 = \begin{bmatrix}
                0 & 2 & 1 \\
                1 & 1 & 1 \\
                1 & 2 & 2 \\
            \end{bmatrix}, \quad 
            N_3 = \begin{bmatrix}
                0 & 2 & 1 \\
                -1 & 1 & 1 \\
                -1 & 2 & 2\\
            \end{bmatrix},
        \end{gathered}
    \end{equation}
    along with the base triple $(1,2,3)^T$, with the additional restriction that $N_1$ produces a valid solution triple with all entries positive if and only if $x_1 > x_3/\sqrt{2}$. The corresponding dynamical system is given by 
    \begin{equation}\label{1/root2Map}
        T_4(t) = \begin{cases}
            \frac{t}{1-2t} & t \in \left(0, \frac{1}{2+\sqrt{2}}\right), \\
            \frac{1}{2t}-1 & t \in \left( \frac{1}{2+\sqrt{2}}, \frac{1}{2} \right),\\
            1-\frac{1}{2t} & t \in \left(\frac{1}{2}, \frac{1}{\sqrt{2}} \right),
        \end{cases}
    \end{equation}
    with the conjugacy map being $\pi(x,y) = \frac{y}{x+1}$, and possesses the $\sigma$-finite, invariant, ergodic and conservative measure
    \begin{equation}\label{1/root2 Measure}
        d\mu = \left(\mathbbm{1}_{(0, 1/(2+\sqrt{2}) )}\cdot\frac{1}{t(1-t)} + \mathbbm{1}_{(1/(2+\sqrt{2}) , 1/\sqrt{2} )}\cdot \frac{1}{t}\right) \, dt .
    \end{equation}
\end{thm}

The above example is distinct from all others, as this map is not fully branched (i.e. not all of the domains on which the function is defined map to the entire interval $[0, 1/\sqrt{2})$). This is due to the condition on $N_1$ -- this is why the invariant measure looks significantly different from the others. However, we make explicit note that it is still a Markov map. 

The final example is a map for which the associated matrix system for generating solutions for the quadratic form is entirely novel, and thus will be the throughline for the paper (since the method of computation is best explained with a working example).

\begin{thm}
    For the quadratic form $Q(\underline{x}) = x_1^2 + x_1x_2 + x_2^2 - 3x_3^2$, the matrix generating system for the solutions in $\{ x_1, x_2, x_3 \in \mathbb{Z} : x_2> x_1 >0, x_3 >0 \}$ consists of the five matrices 
    \begin{equation}\label{root3matrices}
    \begin{gathered}
        N_1 = \begin{bmatrix}
            1 & 0 & 0 \\
            3 & 7 & 12 \\
            2 & 4 & 7 \\
        \end{bmatrix}, \quad 
        N_2 = \begin{bmatrix}
            0 & 1 & 0 \\
            7 & 3 & 12 \\
            4 & 2 & 7 \\
        \end{bmatrix}, \quad 
        N_3 = \begin{bmatrix}
            1 & 1 & 0 \\
            -4 & 3 & 2 \\
            -2 & 2 & 7 \\
        \end{bmatrix}, \\
        N_4 = \begin{bmatrix}
            1 & 1 & 0 \\
            3 & -4 & 12 \\
            2 & -2 & 7 \\
        \end{bmatrix}, \quad 
        N_5 = \begin{bmatrix}
            0 & 1 & 0 \\
            -7 & -4 & 12\\
            -4 & -2 & 7 \\
        \end{bmatrix},
        \end{gathered}
    \end{equation}
    along with the base triples $(2,11,7)^T$ and $(1, 22, 13)^T$. The corresponding 1-dimensional dynamical system (up to homeomorphic conjugacy) is 
    \begin{equation}\label{root3map}
        T_5(t) = \begin{cases}
            \frac{t}{1-t} & t \in \left(0, \frac{\sqrt{3}}{1+\sqrt{3}} \right), \\[0.5em]
            \frac{3-3t}{t} & t \in \left(\frac{\sqrt{3}}{1+\sqrt{3}}, 1 \right)x,\\[0.5em]
            \frac{3-3t}{2t-3} & t \in \left(1, \frac{3+\sqrt{3}}{2+\sqrt{3}} \right),\\[0.5em]
            \frac{3-2t}{t-1} & t \in \left( \frac{3+\sqrt{3}}{2+\sqrt{3}}, \frac{3}{2} \right),\\[0.5em]
            \frac{2t-3}{2-t} & t \in \left(\frac{3}{2}, \sqrt{3} \right),\\
        \end{cases}
    \end{equation}
    with the conjugacy map being  $\pi(x,y) = \frac{y-x}{x+1}$, and possesses the $\sigma$-finite, invariant, ergodic and conservative measure
    \begin{equation} d\mu = \frac{1}{t}\, dt.\end{equation}
\end{thm}

The last system stated will be our point-of-view for the rest of the paper, since this is the only example for which the calculation of the matrices was not already completed in \cite{Cha2018} as an example.

\section{Solution Generating Systems}\label{MatricesSection}
The result of producing Pythagorean triples using the matrix system is well-known. In \cite{Cha2018}, Cha et al.\ set out to define a method of producing these collections of matrices for arbitrary quadratic forms on $\mathbb{R}^3$. Precise details of their method can be found in \cite[Algorithm 3.5]{Cha2018}, and many of the details will be omitted in this section. However, some key details of their method are useful later in this article, in which case these will be stated without proof.

To each quadratic form $Q$, we may associate the symmetric bilinear form 
\begin{equation}
    \langle \underline{x}, \underline{y} \rangle_Q = \frac{1}{2}\left( Q(\underline{x}+\underline{y}) - Q(\underline{x}) - Q(\underline{y}) \right). 
\end{equation}
and additionally, we can associate a symmetric matrix $M_Q$ to this symmetric bilinear form (and consequently to $Q$).
\begin{defn}
    A matrix $A$ is \underline{orthogonal with respect to $Q$} if $\langle A\underline{x}, A\underline{y} \rangle_Q = \langle \underline{x}, \underline{y} \rangle_Q$ for all vectors $\underline{x}, \underline{y} \in \mathbb{R}^3$. Denote the set of all such matrices by $ O_Q(\mathbb{R})$. A matrix is \underline{integral} (w.r.t. a given basis of $\mathbb{R}^3$) if all the entries of the matrix are integers. 
\end{defn}
For shorthand, we write $O_Q(\mathbb{Z})$ for the collection of all matrices that are orthogonal w.r.t. $Q$ and integral. Note that the set of orthogonal matrices is closed under multiplication.

For an integral vector $\underline{w} \in \mathbb{Z}^3$, define the transformation $s_{\underline{w}}$ by 
\begin{equation}
    s_{\underline{w}}(\underline{x}) := \underline{x} - 2\frac{\langle \underline{x}, \underline{w} \rangle_Q}{Q(\underline{w})}\underline{w},
\end{equation}
and note that this transformation is self-inverse. 

Finally, for a form $Q$ of signature $(2,1)$ with non-zero integral solutions to $Q(\underline{x}) = 0$, whose associated bilinear form is non-degenerate, we may define an action of matrices $A \in O_Q(\mathbb{R})$ on affine cross-sections of the surface. Most commonly, the cross-section $\mathcal{E}$ chosen will be defined by $x_3=1$. The appropriate projection map $p_\mathcal{E}$ normalises the $x_3$ coordinate, so 

\[ p_\mathcal{E}((a,b,c)^T) = \left( \frac{a}{c}, \frac{b}{c} \right) \]

though this can be generalised to other projections onto different cross-sections by changing basis. For the third coordinate projection case, writing the matrix $A$ with respect to the standard basis, for any $(x,y) \in \mathcal{E}$, define 
\begin{equation}\label{projectionDef}
    \widehat{A}(x,y) = \frac{1}{r_3(A)\underline{v}}A \underline{v}
\end{equation}
where $\underline{v}$ is any preimage of $(x,y)$ under $p_\mathcal{E}$, and $r_3(A)$ is the third row of $A$. This is well-defined since $A$ is linear on $\mathcal{R}^3$ (so any preimage can be taken) and orthogonal with respect to $Q$ (so $A$ preserves the surface $Q(\underline{v}) = 0$). Again, this will generalise to different projections following the same idea. The hat notation will be used to refer to the appropriate action of $A$ on $\mathcal{E}$ moving forward.

Now, the procedure from \cite{Cha2018} will be discussed. To start, a region $F_Q\subset \mathbb{R}^3$ is chosen according to some conditions -- this will be the region for which the matrix system will generate the integer solutions of. An appropriate projection is also chosen, with $FA_Q$ denoting the projection of \newline$F_Q \cap \{\underline{v} :  Q(\underline{v})=0 \}$. There are some additional properties $FA_Q$ must satisfy, however these are omitted in this article for brevity (see \cite{Cha2018} for more details). 

We now highlight an important step in their construction which gives rise to some useful facts for this article. In \cite{Cha2018}, a vector $\underline{w}$ is chosen so that $Q(\underline{w}) \neq 0$, and that $s_{\underline{w}}$ is integral (which will ensure that the output matrices will be integral). The system of matrices $N_1, ..., N_k$ produced are defined by 
\[ N_i = T_{\underline{w}} U_i \]
where $T_{\underline{w}} = \pm s_{\underline{w}}$ depending on the sign of $Q(\underline{w})$ -- take positive sign if $Q(\underline{w})>0$ and vice versa -- and the $U_i \in O_Q(\mathbb{Z})$ are chosen such that 
\begin{equation}\label{nonIntersecting}
    \left| \widehat{T_{\underline{w}}^{-1}}(FA_Q) \setminus \bigcup_{i=1}^k \widehat{U_i}(FA_Q) \right| < \infty \quad \mathrm{and} \quad \widehat{U_i}(FA_Q) \cap \widehat{U_j}(FA_Q) = \emptyset
\end{equation}
for all $i \neq j$. Notably, this means that $N_i \in O_Q(\mathbb{Z})$ for all $i$. We may also replace $T_{\bm{w}}^{-1}$ with $T_{\bm{w}}$, as $T_{\bm{w}}$ is self-inverse.

The method identifies which are the "base triples", and defines the "Berggren Tree" as the collection of graph-theoretic trees with nodes being integer solutions to the quadratic form in $F_Q$, base nodes being the base triples, and two nodes connected by an edge if one can be reached by another by right-multiplying by some matrix $N_i$. Cha et al.\ \cite{Cha2018} note that this construction is indeed a collection of trees, meaning that it contains no cycles and there is a path between any two given nodes along edges, and therefore any integer solution in $F_Q$ has a unique path along edges to one of the base triples, and only occurs in one of the trees.

\begin{exmp}\label{exmpl1}
    Here is the example presented for illustrative purposes. Let \[Q(\underline{x}) = x_1^2 + x_1x_2 + x_2^2 -3x_3^2 \quad \mathrm{and}\quad F_Q = \{ \underline{x} \in \mathbb{Z}^3 : x_2>x_1>0, x_3>0 \}.\] Note that $(2,11,7)$ is indeed a solution to $Q(\underline{x})=0$ in $F_Q$, so the conditions for the form are satisfied. Setting $\underline{w}= (0,2,1)$, one can compute that $Q(\underline{w}) = 1 > 0$, and thus
    \[T_{\underline{w}} = s_{\underline{w}} = \begin{bmatrix}
        1 & 0 & 0 \\
        -4 & -7 & 12 \\
        -2 & -4 & 7\\
    \end{bmatrix}.\]
    One can then also verify that the collection of matrices 
    \begin{equation*}
    \begin{gathered}
        U_1 = \begin{bmatrix}
            1 & 0 & 0 \\
            -1 & -1 & 0 \\
            0 & 0 & 1\\
        \end{bmatrix}, \quad 
        U_2 = \begin{bmatrix}
            0 & 1 & 0\\
            -1 & -1& 0 \\
            0 & 0 & 1 \\
        \end{bmatrix}, \quad
        U_3 = \begin{bmatrix}
            1 & 1 & 0 \\
            0 & -1 & 0 \\
            0 & 0 & 1 \\
        \end{bmatrix},\\
        U_4 = \begin{bmatrix}
            1 & 1 & 0 \\
            -1 & 0 & 0 \\
            0 & 0 & 1\\
        \end{bmatrix}, \quad
        U_5 = \begin{bmatrix}
            0 & 1 & 0 \\
            1 & 0 & 0 \\
            0 & 0 & 1
        \end{bmatrix},
    \end{gathered}
    \end{equation*}
    does indeed satisfy \eqref{nonIntersecting}, and consequently we may define $N_i = T_{\underline{w}} U_i$ as specified in \eqref{root3matrices}.
\end{exmp}

\section{Coding Solutions and Dynamical Systems}\label{MapsSection}
This section deals with encoding the solutions to the quadratic form $Q$ in the specified subset of $\mathbb{Z}^3$ using the tree of solutions constructed using the matrix generating system. Following the definition of codes for points, a natural resulting problem is that of computing codes for given solutions. This is solved by noting a bijection exists between the rational points of the ellipsoid $\mathcal{E} \subset \mathbb{R}^{2}$ created by dividing through by the "final coordinate", as illustrated in Section \ref{MatricesSection}. The resulting map is a 1-dimensional endomorphism on $\mathcal{E}$, and can then be conjugated to an interval map by performing a stereographic projection from some other chosen point on the ellipsoid.

From here, assume that the matrix generating system (and therefore Berggren tree) has been computed, and that the quadratic form is in 3 variables. This means that any triple integer solution exists as a vertex in the Berggren tree, with a unique path along edges to one of the base triples. Equivalently, for any solution $(a,b,c)^T \in F_Q \subseteq \mathbb{Z}^3$ (where $F_Q$ is the region of integer solutions we are interested in), we have a unique representation 
\begin{equation}\label{TripleMatrixRep}
    \begin{bmatrix}
        a \\ b \\ c
    \end{bmatrix} = M_{a_1}M_{a_2} \cdots M_{a_k} \underline{b}
\end{equation}
where $\underline{b}$ is one of the base triples, and $M_i$ are the matrices that generate the integer solutions, both found via the method in Section \ref{MatricesSection}. By uniqueness, one can identify the triple $(a,b,c)^T$ with the sequence $(a_1, \ldots, a_n)$, and by putting restrictions on the space of allowed sequences, this identification is a bijection.

Now, define an algorithm that computes the sequence for a given triple. First, recall that there exists a bijection between primitive integer solutions to $Q(\underline{x})$ (i.e. there is no common divisor between $a,b$ and $c$) and rational points on the ellipse $\mathcal{E}$. Define $FA_Q \subseteq \mathcal{E}$ to be the arc of $\mathcal{E}$ whose rational points correspond to triples in $F_Q$. Recall the definition of the action of $\widehat{M}$ (for orthogonal matrices $M \in O_Q(\mathbb{R})$ as presented in \eqref{projectionDef}). Note again that this can be generalised to different "normalisation" projections by a change of coordinate basis.

Recall additionally that the images $\widehat{U_i}(FA_Q)$ were disjoint (see \eqref{nonIntersecting}). To each $M_i$, there is an associated subset $D_i \subset FA_Q$ such that $\widehat{M_i}(D_i) \subset FA_Q$. One may write this $D_i$ explicitly as 
\[ D_i := \widehat{U_i^{-1}}( \widehat{T_{\bm{w}}}(FA_Q) \cap \widehat{U_i}(FA_Q)),\]
whence the image of $D_i$ under $\widehat{M_i}$ may be computed to be
\[ \widehat{M_i}(D_i) = FA_Q \cap \widehat{M_i}(FA_Q).\]
By both properties of \eqref{nonIntersecting}, these $\widehat{M_i}(D_i)$ are also disjoint, and they completely tile $FA_Q$ except for at most finitely many points. Consequently, for any rational point in $FA_Q$, the leftmost matrix in the representation \eqref{TripleMatrixRep} is uniquely determined by the region of $FA_Q$ the rational point lies in.

Construct a map from $\mathcal{E}$ to itself as follows. Let $\mathcal{D}_i := \widehat{M_i}(D_i)$. Since these form a partition of $FA_Q$, we may define the map $\widehat{T}$ by
\[ \widehat{T}(\underline{x}) = \widehat{M_i^{-1}}(\underline{x}) \ \mathrm{for} \ \underline{x} \in \mathcal{D}_i\cap \mathbb{Q}^2.\]
Then $\widehat{T}$ is well-defined (for every rational point except at most finitely many), and under the identification of rational points in $FA_Q \subset \mathcal{E}$ with matrix products of the form \eqref{TripleMatrixRep}, acts like
\begin{equation}
     M_{a_1}M_{a_2} \cdots M_{a_k} \underline{b} \longmapsto M_{a_2}M_{a_3} \cdots M_{a_k} \underline{b},
\end{equation}
i.e. by a shift map on the space of sequences. Then, returning the code of a rational point in $S$ using $\widehat{T}$ is a matter of recording which regions the orbit of the point lands in. More explicitly, let 
\begin{equation}
    d(\underline{x}) = \begin{cases}
        i & \mathrm{if}\ \underline{x} \in \mathcal{D}_i, \\
        \underline{b} & \mathrm{if} \ \underline{x} = \underline{b}, \ \mathrm{and} \ \underline{b} \ \mathrm{is \ a \ base \ triple}, \\
        0 & \mathrm{else}.
    \end{cases}
\end{equation}
Then the sequence associated with $\underline{x}$ via the identification \eqref{TripleMatrixRep} is exactly 
\[ \left(d\left(\widehat{T^i}(\underline{x})\right)\right)_{i=0}^{\infty}\ .  \]
Conjugating this to a map on some bounded interval in $\mathbb{R}$ is a matter of finding an appropriate stereographic projection from some other chosen point in $\mathcal{E}$. There are many options of projection points which result in maps on bounded intervals, but the ones presented in this article are chosen so that the interval is of the form $(0, t)$ for some positive $t$. 

\begin{exmp}
    Return to the setting of Example \ref{exmpl1} with $Q(\underline{x}) = x_1^2 + x_1x_2 + x_2^2 - 3x^3 $, and the collection of matrices that generate all solutions in $\{ \underline{x} \in \mathbb{R}^3 : x_2>x_1 > 0, x_3>0\}$ being
    \begin{equation}
    \begin{gathered}
        N_1 = \begin{bmatrix}
            1 & 0 & 0 \\
            3 & 7 & 12 \\
            2 & 4 & 7 \\
        \end{bmatrix}, \quad 
        N_2 = \begin{bmatrix}
            0 & 1 & 0 \\
            7 & 3 & 12 \\
            4 & 2 & 7 \\
        \end{bmatrix}, \quad 
        N_3 = \begin{bmatrix}
            1 & 1 & 0 \\
            -4 & 3 & 2 \\
            -2 & 2 & 7 \\
        \end{bmatrix}, \\
        N_4 = \begin{bmatrix}
            1 & 1 & 0 \\
            3 & -4 & 12 \\
            2 & -2 & 7 \\
        \end{bmatrix}, \quad 
        N_5 = \begin{bmatrix}
            0 & 1 & 0 \\
            -7 & -4 & 12\\
            -4 & -2 & 7 \\
        \end{bmatrix},
        \end{gathered}
    \end{equation}
\end{exmp}
with base triples $(1,22,13)$ and $(2,11,7)$. The ellipse given by normalising the last coordinate is $\mathcal{E} = \{ (x,y) : x^2 + xy + y^2 = 3\}$, with $FA_Q = \{ (x,y) \in \mathcal{E} : y>x>0\}$. Constructing the associated map $\hat{T}$ as outlined yields
\begin{equation*}
    \hat{T}(x,y) = \begin{cases}
        \left( \frac{x}{1-2x-4y}, \frac{3x+7y-12}{7-2x-4y}\right) & \mathrm{if}\ y/x > 22, \\
        \left( \frac{3x+7y-12}{7-2x-4y}, \frac{x}{7-2x-4y} \right) & \mathrm{if}\ 22> y/x > 4\sqrt{3} + 3,\\
        \left( \frac{12-3x-7y}{7-2x-4y}, \frac{4x+7y-12}{7-2x-4y} \right) & \mathrm{if}\ 4\sqrt{3} + 3 > y/x > 11/2,\\
        \left( \frac{4x+7y-12}{7-2x-4y}, \frac{12-3x-7y}{7-2x-4y} \right) & \mathrm{if} \ 11/2 > y/x > 4\sqrt{3} - 4 ,\\
        \left(\frac{12-4y-7x}{7-2x-4y}, \frac{x}{7-2x-4y} \right) & \mathrm{if}\  4\sqrt{3} - 4 > y/x > 1.
    \end{cases}
\end{equation*}
To conjugate this to an interval map, we can stereographically project onto the $y$-axis from the point $(-1,-1) \in \mathcal{E}$. The explicit formula for this is
\[ \pi(x,y) = \frac{y-x}{x+1}, \quad \pi^{-1}(t) = \left(\frac{3-t^2}{t^2 + 3t + 3}, \frac{2t^2 + 6t + 3}{t^2 + 3t + 3} \right),\]
which yields an interval map on $[0, \sqrt{3})$ with formula
\begin{equation*}
    T(t) = \begin{cases}
        \frac{t}{1-t} & t \in \left(0, \frac{\sqrt{3}}{1+\sqrt{3}} \right),\\
        \frac{3-3t}{t} & t \in \left(\frac{\sqrt{3}}{1+\sqrt{3}}, 1\right), \\
        \frac{3-3t}{2t-3} & t \in \left(1, 3-\sqrt{3} \right), \\
        \frac{3-2t}{t-1} & t \in \left( 3-\sqrt{3}, \frac{3}{2}\right), \\
        \frac{2t-3}{2-t} & t \in \left(\frac{3}{2}, \sqrt{3} \right).\\
    \end{cases}
\end{equation*}
For illustrative purposes, a graph of $T$ is presented below.
    
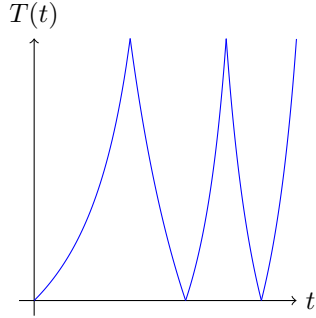
\begin{center}
        \begin{tikzpicture}
  \draw[->, scale=2] (-0.1, 0) -- (1.73205, 0) node[right] {$t$};
  \draw[->,scale=2] (0, -0.1) -- (0, 1.73205) node[above] {$T(t)$};
  \draw[scale=2, domain=0:0.63397, smooth, variable=\x, blue] plot ({\x}, {\x/(1-\x)});
  \draw[scale=2, domain=0.63397:1, smooth, variable=\x, blue] plot ({\x}, {(3-3*\x)/(\x)});
  \draw[scale=2, domain=1:1.26795, smooth, variable=\x, blue] plot ({\x}, {(3-3*\x)/(2*\x-3)});
  \draw[scale=2, domain=1.26795:1.5, smooth, variable=\x, blue] plot ({\x}, {(3-2*\x)/(\x-1)});
  \draw[scale=2, domain=1.5:1.73205, smooth, variable=\x, blue] plot ({\x}, {(2*\x-3)/(2-\x)});
\end{tikzpicture}
    \end{center}

\section{Building Extensions and Finding Invariant Measures}\label{ExtensionSection}

We now build an extension of the solution map following Keane's method in \cite{Keane95}, i.e. forcing the extension to be linear. This way, verifying that Lebesgue measure is an invariant measure of the extension (defined on the appropriate domain) is a matter of computing determinants. As it turns out, many of the steps in Keane's paper remain untouched, in particular calculating the push-forward of the extension's invariant measure onto the original system via the appropriate projection is identical. Thus, every invariant measure of some map formed in this manner is exhibited as the integral of the same density over some region of $\mathbb{R}$, which is determined by the domain of the linear extension. 

There is much interest in computing natural extensions (or indeed even just extentions) of number theoretic maps. Dajani, Kraaikamp and Solomyak \cite{Dajani96} detail a method of calculating the invariant ergodic measures of the $\beta$ transformation by explicitly constructing a particular presentation of the natural extension of the map $x \mapsto \beta x $ mod 1, which also happens to be Lebesgue-preserving (in the appropriate component). In this article, we are not necessarily interested in verifying that the extension we build is the natural extension -- in order to compute invariant measures, it is only necessary to construct a normal extension -- however a particular step in Keane's construction \cite{Keane95} can be viewed as the natural extension of the Gauss map, so it may not be difficult to show that the same is true in our adapted case. This could potentially be observed by noting the adapted Keane method appears similar to a form of the De Sitter space associated with the piecewise Mobius transformations \cite{panti2020}. The De Sitter space is a representation of the natural extension of discrete groups of Mobius transformations, which provides a link between invariant measures of piecewise Mobius transformations and geodesic flows on hyperbolic surfaces. The following construction which is an effort to generalise Keane's method is one way of computing a representation of this space. For this article, we do not go more into detail on this subject, as the connection (while strong) has not been mathematically verified, and the construction below is more explicit and perhaps more mechanically illuminating.

More specifically, there has been much work conducted in the construction of natural extensions of piecewise M\"{o}bius transformations (see \cite{arnoux17, calta24,PANTI202250}), so it should be explicitly stated that this line of investigation is definitely not novel. However, all of said papers and articles do not mention these specific maps as examples, nor do they draw connections to maps derived from matrix systems that generate solutions for quadratics like in \cite{Cha2018} (with the notable exception of \cite{PANTI202250} that mentions only the Romik map, and deals predominantly with transformations arising from Hecke group actions). Furthermore, while \cite{calta24} does build Lebesgue-preserving extensions (albeit in a slightly different manner to how this article does), both it and \cite{arnoux17} are primarily concerned with proving results for probability preserving transformations, and omit much discussion that could be had about transformations with infinite invariant measures, of which the maps in Section \ref{Main Results} are examples.

\subsection{Maps with Full Branches}\label{fullbranchcase}

First, we define the action of a piecewise M\"{o}bius transformation on the space of quadratic forms in one variable. Suppose the piecewise transformation $T : X \rightarrow X$, where $X\subset \mathbb{R}$, consists of full branches (i.e. each branch maps its domain to $X$ Lebesgue-almost entirely). These branches will be denoted by $\phi_i(x) = \frac{px+q}{rx+s}$, and defined on $D_i$, with index set $i \in I \subseteq \mathbb{N}$. Assume without loss of generality that $ps-rq = \pm1$. For any $\alpha \in D_i$, we have that there is a unique choice of $\beta \in X$ such that $\alpha = \phi_i^{-1}(\beta)$. Additionally, for any $\alpha \in X$, there exist real numbers $a,b,c$ such that the quadratic polynomial $p(x) = ax^2 + bx -c $ has two distinct real roots, with one of them equal to $\alpha$. Via some algebraic manipulation, we see that 

\begin{align}
    p(\alpha) = a\alpha^2 + b\alpha - c = a \left( \frac{s\beta-q}{-r\beta+p} \right)^2 + b \frac{s\beta - q}{-r\beta + p} - c &= 0 \nonumber \\
    \implies a(s\beta - q)^2 + b(s\beta - q)(-r\beta + p) - c(-r\beta + p)^2 &= 0 \nonumber \\
    \implies \beta^2( as^2 -brs -cr^2) + \beta(-2asq + b(ps+qr) + 2cpr) - (-aq^2 + bqp + cp^2) &= 0.\label{algebraic manip}
\end{align}

Thus, define 

\begin{equation}\label{phi action}
   \overline{p}(x) := x^2( as^2 -brs -cr^2) + x(-2asq + b(ps+qr) + 2cpr) - (-aq^2 + bqp + cp^2),
\end{equation}

and define the action of $\phi_i$ on the space of quadratic forms to send $p$ to $\overline{p}$. This way, by the construction of $\overline{p}$, the original dynamics are preserved in the location of one of the roots of the quadratic. Notice that we could equally have chosen $-\overline{p}$ by simply moving the entire equation to the RHS of \eqref{algebraic manip} -- this option will be eliminated later. 

We may represent the actions in terms of the coefficients of the polynomials $p$ and $\overline{p}$, which can be easily observed to be a linear function, and therefore represented by a matrix. For this $\phi_i$, denote the action on quadratic forms by $\overline{\phi}_i$, whose matrix is given by
\begin{equation}\label{MatrixAction}
M(\overline{\phi}_i) = \begin{bmatrix}
    s^2 & -rs & -r^2 \\
    -2sq & ps+qr & 2pr \\
    -q^2 & qp & p^2 \\
\end{bmatrix}
\end{equation}
with determinant $(ps - qr)^3 = (\pm 1)^3$, thus this linear map preserves Lebesgue measure. The same is true for every branch $\overline{\phi}_i$. Furthermore, via similar algebraic manipulation, we can see that each $\overline{\phi}_i$ preserves the discriminant of the polynomial, i.e. \[\Delta(p) = b^2 + 4ac = b^2(ps-qr)^2 + 4(ps-qr)^2ac=\Delta(\overline{p}).\]

The final and most important step is to now construct restrictions on $a,b,c \in \mathbb{R}$ such that this map is well-defined and bijective. Firstly, we demand that $a>0$ -- this is to establish that we only have one correct choice between $\overline{p}$ and $-\overline{p}$ (exactly which of these is chosen is not important, only that we have exactly one output). Next, demand that the second root of the polynomial $p$ lies outside of $X$ -- this is to ensure that the collection of maps $\overline{\phi}_i$ have non-overlapping domains, as the domain of each $\overline{\phi}_i$ will be all the polynomials with exactly one root in $D_i$.

Thirdly, we look at how each map $\overline{\phi}_i$ moves the second root of $p$.
Suppose that $\Tilde{\alpha}$ is the other root of $p$. We may write $\Tilde{\alpha} = \phi_i^{-1}(\Tilde{\beta})$, implying that 
\[ \Tilde{\beta} = \phi_i(\Tilde{\alpha}) = \frac{p\Tilde{\alpha} + q}{r\Tilde{\alpha}+s}. \]
If we find some subset $R \subseteq \mathbb{R}\setminus X$ satisfying the disjoint union property
\begin{equation}\label{DisjointUnion} R = \bigsqcup_{i \in I} \phi_i(R) \quad \mathrm{mod} \ \lambda, \end{equation}
then it immediately follows that the collection $\overline{\phi}_i$, $i \in I$ forms a Lebesgue-a.e.\ bijection on the space of quadratic forms that have exactly two roots such that $\alpha \in X$ and $\Tilde{\alpha} \in R$. This motivates the following definition.

\begin{defn}
Let $\Sigma \subseteq \mathbb{R}^3$ be defined as
\[ \Sigma = \left\{ ax^2 + bx - c : a>0 \ \mathrm{and\ } \left(\frac{-b + \sqrt{b^2 + 4ac}}{2a}, \frac{-b - \sqrt{b^2 + 4ac}}{2a}\right) \in (X \times R) \cup (R \times X) \right\}, \]
i.e. the set of quadratic polynomials with positive $x^2$ coefficient and exactly one root in $R$ and one root in $X$.
\end{defn}

From here onwards, for any polynomial $p\in \Sigma$, when writing its roots $(\alpha, \beta)$ as a tuple, we make the unambiguous choice that $\alpha \in X$ and $\beta \in R$. We have now proven the following theorem.

\begin{thm}
Suppose that we have a continuous piecewise  M\"{o}bius map $T:X \rightarrow X$ defined by $\phi_i$ on $D_i \subset X \subset \mathbb{R}$, i.e.
\[ T|_{D_i} = \phi_i,\]

and that there exists a region $R$ satisfying \eqref{DisjointUnion}. Then the map $\overline{T}:\Sigma \rightarrow \Sigma$ defined as $\overline{\phi}_i$ on the collection of $p\in \Sigma$ that have a root in $D_i$ is a Lebesgue preserving almost-everywhere bijection on $\Sigma$ as parameterised by the coefficients of the polynomial, and is an extension of $T$.
\end{thm}
\begin{proof}[\underline{Proof}]
    As stated above, most of the proof has already been completed, however we make special mention of how $\overline{T}$ can be seen as an extension of $T$. 

    The collection of quadratic polynomials $\Sigma$ can be parameterised in one of two ways: by the location of roots in tandem with the discriminant, or by the three coefficients of the polynomial. The map looks different depending on which parameter space we choose to act on, and our theorem only makes the claim that $\overline{T}$ is Lebesgue preserving on the coefficient parameter space, but the map is still well-defined on both spaces. As a consequence of this, it is easy to see how the original dynamics exists as a factor: observe how the root of $X$ is moved.
\end{proof}

It it from this point onward that the method in \cite{Keane95} is followed almost identically -- using the discriminant preserving property, we may restrict $\overline{T}$ to $\Sigma_{\Delta = 1} = \{ (a,c) : (a, \sqrt{1-4ac}, c) \in \Sigma \}$, setting $b = \sqrt{1-4ac}$ and therefore projecting $\overline{T}$ to a function $\overline{T}_\Delta$ in $a,c$ alone. Then using some differential geometry, we confirm that $\frac{1}{\sqrt{1-4ac}} \, dadc$ is an invariant density of $\overline{T}_\Delta$ -- in the case where we have the Gauss map as the base dynamical system, this is the dynamical system that can be viewed as the natural extension of the Gauss map.

Finally, observing that  
\[ ax^2 + \sqrt{1-4ac}x - c = 0 \iff 1 = \left( \frac{c}{x} + ax \right)^2,\]
we may construct simultaneous equations for the two roots $\omega$ and $\eta$:
\begin{align*}
    \frac{c}{\omega} + a\omega &= -1, \\
    \frac{c}{\eta} + a\eta &= 1,
\end{align*}
(assuming both are equal to 1 results in a contradiction, similarly for assuming both equal $-1$), which can be rearranged to find that 
\begin{equation}\label{ac equations}
    c = \frac{\omega\eta}{\omega - \eta}, \quad a = \frac{1}{\eta - \omega}.
\end{equation}
Note: this is mildly different to \cite{Keane95}, though the idea is the same. Using \eqref{ac equations}, we may compute the invariant density $\frac{1}{\sqrt{1-4ac}}\, dadc$ in terms of $\omega, \eta$ using a transformation of variables to be
\begin{equation}
    \frac{1}{(\eta - \omega)^2}\, d\omega d\eta .
\end{equation}
Observing the square in the denominator, it follows that it doesn't matter whether $\alpha \in X$ is $\omega$ or $\eta$ in \eqref{ac equations}, as the invariant density in terms of $\omega, \eta$ will not change. Finally, by integrating out over the root $\Tilde{\alpha} \in R$, the result will be the invariant density of $T$. This is collated in the following.

\begin{cor}\label{invDensityforT}
    For any continuous piecewise M\"{o}bius transformation $T: X \rightarrow X$ defined with countably many branches $\phi_i$ with all branches full, if there exists some $R \subset \mathbb{R}\setminus X$ such that 
    \[ R = \bigsqcup_{i \in I} \phi_i(R),\]
    then an absolutely continuous invariant measure for $T$ is given by
    \begin{equation}\label{invMeasureFullBranch}
\left(\int_R \frac{1}{(\eta - \omega)^2} \, d\eta \right) \, d\omega.
\end{equation}
\end{cor}

\begin{rem}
    This corollary is reminiscent of Panti \cite[Theorem 3.4]{PANTI202250}. In particular, Panti deals with maps constructed from Hecke group considerations, and computes natural extensions and the invariant measures for said extensions. The origin of the maps in this article are not based in Hecke groups, though it is noted that the Romik map can be viewed in that manner, so it is not implausible that with some clever choices of conjugacy one could view many of these maps derived from solution generating systems similarly. However, it is not clear (at least to the author) how to determine the correct Hecke group and/or conjugacy for a given Berggren tree and matrix system.
\end{rem}

\begin{exmp}\label{exmplInvariantDensity} 
The following is an illustration of this method for the map $T_5$ as defined in \eqref{root3map}. For convenience, ennumerate the branches of the map from top to bottom (i.e. first branch is $\phi_1$, and so on). The exact matrices will not be calculated here, as Corollary \ref{invDensityforT} notes that explicit construction of the extension is unnecessary, provided that the only result of interest is the measure. Note that $R = (-\infty, 0)$ indeed satisfies \eqref{DisjointUnion}:
\begin{align*}
    \phi_1(t) = \frac{t}{1-t} &\implies \phi_1(R) = \left( -1, 0 \right), \\
    \phi_2(t) = \frac{3-3t}{t} &\implies \phi_2(R) = \left( -\infty , -3\right), \\
    \phi_3(t) = \frac{3-3t}{2t-3} &\implies \phi_3(R) = \left( -\frac{3}{2}, -1 \right), \\
    \phi_4(t) = \frac{3-2t}{t-1} &\implies \phi_4(R) = \left( -3, -2 \right), \\
    \phi_5(t) = \frac{2t-3}{2-t} &\implies \phi_5(R) = \left( -2, -\frac{3}{2} \right).\\
\end{align*}
Consequently, $\int_{-\infty}^0 \frac{1}{(\eta - \omega)^2} d\eta = \frac{1}{\omega}$ is an invariant density for \eqref{root3map}.
\end{exmp}

\subsection{Non-Full Branches}
Not every example presented in Section \ref{Main Results} has the full-branch property. The construction of the Lebesgue-preserving extension is still possible, however the conditions that $R$ satisfies must be adapted slightly. This sub-section presents the appropriate generalisation.

Assume that, instead of each M\"{o}bius branch of the transformation $T:X \rightarrow X$ mapping $D_i$ to the entire interval $X$, the branch $\phi_i$ maps $D_i$ to the union of some collection of $D_j$'s. This implies that $T$ is a Markov map with $\{D_i\}_{i \in I}$ providing an appropriate Markov partition, where $I$ is some index set. The appropriate sequence space for which the shift map is conjugate to $T$ can therefore be defined in terms of forbidden words of length 2, or equivalently is defined in terms of which words of length 2 are valid. Let the collection of valid words of length 2 be $W^+$.

To ensure that the constructed extension on the space of appropriate quadratic forms is well-defined and bijective, the region $R$ in which the other roots of the polynomials $p$ reside must satisfy the following: 
\begin{itemize}
    \item $R \subseteq \mathbb{R} \setminus X$, so that the map is well-defined, \\
    \item There exists a partition of $R$ into $\{R_i\}_{i \in I}$ such that 
    \begin{equation}\label{markovMapCondition}
        \phi_i^{-1}(R_i) = \bigcup_{\{j \ :\ ji \in W^+\}} R_j \quad \mathrm{mod} \ \lambda,
    \end{equation}
    so that the map is Lebesgue-a.e.\ bijective.
\end{itemize}
Additionally, impose the restriction on $\Sigma$ (the space of polynomials) that every polynomial has two distinct real roots, say $\omega$ and $\eta$, satisfying
\begin{equation}\label{pRootsCondition}
\omega \in D_i, \eta \in R_j \implies ji \in W^+.\end{equation}
This aids in ensuring the almost-everywhere bijectivity of the extension, as then the "other root" of the polynomials are always mapped into an appropriate partition element of $R$ by the action of the extension.

Again, most of the analysis from Section \ref{fullbranchcase} remains intact, with the exception that integrating out the roots of the polynomials not in $X$ is no longer simply integrating over the entirety of $R$. Instead, perform the integration piecewise with respect to the partition $\{D_i\}_{i \in I}$. Letting 
\[ R(D_i) := \bigcup_{\{j\ :\ ji \in W^+ \}} R_j,\]
the invariant measure for $T$ will now be described by
\begin{equation}
\sum_{i \in I}\left( \mathbbm{1}_{D_i}(\omega) \cdot \int_{R(D_i)} \frac{1}{(\eta - \omega)^2} \, d\eta \right) \, d\omega.
\end{equation}

\begin{rem}
    In the event all branches are full, this formula exactly returns \eqref{invMeasureFullBranch}, so this is (in some sense) truly a generalisation of the previous process.
\end{rem}

\begin{exmp}
    Using the map $T_4$ as an example (see \eqref{1/root2Map}), set $R = (-\infty, 0) \cup (1, \infty)$ and define an indexed partition by 
    \[ R_1 = (-1, 0), \quad R_2 = (-\infty, -1), \quad R_3 = (1, \infty),\]
    Enumerate the branches from top to bottom again, and label the domain of branch $\phi_i$ by $D_i$. Note that the appropriate restriction on the sequence space forbids the words "$32$" and "$33$" (which can be seen by observing that $\phi_3$ maps its domain onto the domain of $\phi_1$, or alternatively that the $N_1$ matrix which corresponds to the action $\phi_3^{-1}$ only produces valid edges for certain triples). Then it can be verified that 
    \[ \phi_1^{-1}(R_1) = R = R_1 \cup R_2 \cup R_3, \quad \phi_2^{-1}(R_2) = \phi_3^{-1}(R_3) = ( - \infty , 0) = R_1 \cup R_2,\]
    and so \eqref{markovMapCondition} is indeed satisfied. Therefore some computation of integrals yields that the invariant measure has the form 
    \begin{align}\label{1/root2 density}
        d\mu &= \left[\mathbbm{1}_{D_1}(t) \cdot \int_{R_1 \cup R_2 \cup R_3}\frac{1}{(\eta - t)^2}\, \mathrm{d}\eta + \mathbbm{1}_{D_2\cup D_3}(t) \cdot \int_{R_1 \cup R_2}\frac{1}{(\eta - t)^2} \, \mathrm{d}\eta  \right] \, \mathrm{d}t \nonumber \\
        &= \left(\mathbbm{1}_{D_1}(t)\cdot\frac{1}{t(1-t)} + \mathbbm{1}_{D_2 \cup D_3}(t)\cdot \frac{1}{t}\right) \, dt,
    \end{align}
    as anticipated.
\end{exmp}

Recall that the main results of the paper are concerning constructing examples of these 1-D dynamical systems associated with solution generating systems (and their invariant measures). Currently, no claim is made about whether this process works in generality, however we make the following conjecture in this vein.

\begin{conj}
    Whenever the Cha-Nguyen-Tauber procedure produces a solution generating system, one may construct a presentation of the solution generating system that is a 1-D piecewise M\"{o}bius transformation. Additionally, this map possesses a $\sigma$-finite, absolutely continuous and invariant measure, which makes the system conservative and ergodic.
\end{conj}

\underline{Partial Proof}: The construction of a 1-D piecewise Mobius transformation is not too difficult to see. One may recall that the definition of the hyperboloid model of hyperbolic geometry looks at the positive sheet of the curve $\{ (x,y,z) \in \mathbb{R}^3 : x^2 + y^2 -  z^2 = -1\}$ equipped with the metric 
    \[ d_H (\underline{x}, \underline{y}) = \arccosh\left(-\langle \underline{x}, \underline{y} \rangle \right),\]
    where $\langle -, - \rangle $ is the symmetric bilinear form associated with the quadratic form $x^2 + y^2 - z^2$. One may analogously construct another model of the hyperbolic plane by using a different non-degenerate form of signature $(2,1)$: we may observe the set $\{ Q(\underline{x}) = -1\}$ equipped with the metric 
    \[ d_H(\underline{x}, \underline{y}) = \arccosh \left( -\langle \underline{x}, \underline{y}\rangle_Q\right),\]
    where $\langle-,-\rangle_Q$ is the associated symmetric bilinear form as defined previously. As any such form is diagonalisable, one can see that the process of diagonalising (and then normalising in the coordinate directions if necessary) defines an isometry between the standard hyperboloid model of the hyperbolic plane and the new model with the above metric. This process also identifies which is the "positive sheet" of the surface $\{ Q(\underline{x}) = -1\}$.
    
    In the standard hyperboloid model, the isometries are exactly the positive orthogonal matrices $PO(2,1)$ (i.e. orthogonal matrices that preserve the positive light cone). Note that these are not necessarily orientation preserving. Conjugating any element of $PO(2,1)$ through the isometries onto the new model gives a matrix that is orthogonal with respect to $Q$ and that preserves the identified "positive sheet". 
    
    The process of Cha-Nguyen-Tauber requires that the quadratic form $Q$ is of signature $(2,1)$, is non-degenerate, and the matrices selected to form the solution generating system must obey these properties. Due to this, we have that each matrix is an isometry of the new model of the hyperbolic plane. Consequently, when we compose the appropriate standard isometries of hyperbolic plane models to find an isometry that takes $\{ Q(\underline{x}) = -1\}$ to the upper half plane model, the conjugated action of the matrix under this isometry must give a M\"{o}bius transformation of determinant $\pm 1$, and we are done.

    The remainder of the conjecture is yet to be proven, however it is likely that the De Sitter space will be very useful in computing the proposed absolutely continuous invariant measures. It may also be useful for ergodicity (and conservativity if necessary) of the systems, though that remains to be seen. It must also be noted that the isometry described above that maps $\{Q(\underline{x}) = -1\}$ onto the upper-half plane is not necessarily the same as the particular projection chosen in the computed examples (which were chosen so that we have a map on an interval $[0,a]$ and such that rationality of points is preserved).

\subsection{Ergodic Properties}
Now that invariant measures have been established, standard techniques from infinite ergodic theory can be utilised to prove ergodicity and conservativity. Romik uses these techniques in \cite{Romik2004} for the ergodicity and conservativity of the Romik map, and the methodology presented in this section is much the same. Similar to Romik's case, explicit calculation is omitted as it is easily verified independently, and this article only presents the method for one of the maps presented, since the method is very similar for the other cases.

We choose the map $T_4$ (defined in \eqref{1/root2Map}) to study as the example here, since this being non-full branched causes it to be the most arduous case. The methodology for the other examples is essentially the same. As is often standard in infinite ergodic theory, first conservativity must be shown, so that inducing is well-defined. 

Recall the definition of $T_4$:
\begin{equation*}
        T_4(t) = \begin{cases}
            \phi_1(t) = \frac{t}{1-2t} & t \in \left(0, \frac{1}{2+\sqrt{2}}\right), \\
            \phi_2(t) = \frac{1}{2t}-1 & t \in \left( \frac{1}{2+\sqrt{2}}, \frac{1}{\sqrt{2}} \right),\\
            \phi_3(t) = 1-\frac{1}{2t} & t \in \left(\frac{1}{2}, \frac{1}{\sqrt{2}} \right),
        \end{cases}
    \end{equation*}
and the density of the absolutely continuous invariant measure $\mu$ as given in \eqref{1/root2 density}.

Formally, this will be considered as a map on the set of points in $\left(0 , \frac{1}{\sqrt{2}} \right]$ that are not pre-images of $0$ under $T_i$, i.e. an automorphism on
\[  X := \left( 0 , \frac{1}{\sqrt{2}} \right] \setminus \left( \bigcup_{i=0}^\infty T^{-i}\left( \{0\} \right) \right). \]
However, since the set of removed points have Lebesgue measure 0, for ease of notation this article will often use intervals of the form $(a,b)$ to refer to $X \cap (a,b)$, and likewise for subsets $A \subset \left(0, \frac{1}{\sqrt{2}}\right]$.

As per Section \ref{MapsSection}, set $[1]:= \left(0, \frac{1}{2+\sqrt{2}}\right]$, $[2] := \left[ \frac{1}{2+\sqrt{2}}, \frac{1}{\sqrt{2}} \right)$ and $[3] :=  \left(\frac{1}{2}, \frac{1}{\sqrt{2}} \right]$. Sequences in the alphabet $\{1,2,3\}$ may be associated to points $x \in X$, and in order to have this identification be a bijection, the additional restrictions imposed on the sequence space are that no sequence may end in an infinite series of 1's, and that the words "32" and "33" are forbidden (the forbidden word conditions are immediate from the definition of $T_4$, and the "infinite 1's" condition follows by observing that $0\in \overline{[1]}$ and is a fixed point of $T_4$).

Define the following set:
\begin{equation}
    J:= ([2]\cup[3])\cap X \neq \emptyset.
\end{equation}
Then, using the invariant measure $\mu$ as given by \eqref{1/root2 Measure}, $J$ has finite measure, and every admissible sequence must be a member of the union $\bigcup_{k=0}^\infty T_4^{-k}(J)$ -- the elements of $X$ not already in $J$ are those beginning with at least one 1, but every sequence of 1's in $X$ must be finite. Consequently, 
\[ X = \bigcup_{k=0}^\infty T_4^{-k}(J) \quad \mathrm{mod}\ \mu, \]
and by Maharam's recurrence theorem (see \cite[Theorem 1.1.7]{Aaronson97}), $T_i$ is conservative.

Now, construct the induced map on $J$. Set $n_J(t) := \min\{n \in \mathbb{N} : T_4^{n} \in J \} $, and define for all $t \in J$,
\begin{equation*} 
T_J(t) = T_4^{n_J(t)}(t) \quad \mathrm{and} \quad 
\mu_J(A) = \mu(A\cap J),
\end{equation*}

which is well defined since within the conservativity proof, it is also shown that $J$ is a sweep-out set. Note that the action of $T_J$ on the space of admissible sequences is some number of left-shifts until the first time that the first digit of the image not equal to 1 (i.e. $(2,1,1,1,3,1,2,3,...)$ is mapped to $(3,1,2,3,...)$). Define a partition of $J$ by
\[ \mathcal{P} := \{ [21^k2], [21^k3], [31^{s}2], [31^s3] : k \geq 0, s\geq1 \}.\]
Observing that $[2] = \bigcup_{k=0}^\infty([21^k2]\cup[21^k3])\ a.e.$ and $[3] = \bigcup_{s=1}^\infty([31^s2] \cup [31^s3] )$, it is clear that $\mathcal{P}$ is a Markov partition for $T_J$. On the elements of the partition, $T_J$ has one of the following forms:
\[ (\phi_1)^k \circ \phi_2(t) = \frac{1-2t}{2t(1+2k)-2} \quad (\phi_1)^s\circ\phi_3(t) = \frac{2t-1}{2t(1-2s)+2s}.\]
Note that these branches of $T_J$ are strictly monotonic on the elements of the partition, and extend to twice continuously differentiable functions on the closure of the partition elements.

It can then be verified by direct calculation that $T_J$ is expanding with bounded distortion, i.e. that (respectively)
\[ \inf_{J}{(T_J)'} > 1 \quad \mathrm{and} \quad \sup_J{\frac{|(T_J)''|}{((T_J)')^2}} < \infty,\]
though care must be taken to restrict the branches of $T_J$ to the partition elements on which they are defined to return the bounded distortion property. With these properties satisfied, we have that $T_J$ is an expanding $C^2$ Markov interval map, which means we may use the results of \cite[Proposition 4.3.3]{Aaronson97}\footnote{This proposition implies that $T_J$ has a sufficiently nice distortion property, but the exact statement requires some more notation and is not illuminating for the proof, so is omitted here.}.

Since $T_J$ is also $\mu_J$-preserving, and $\mu_J$ is finite (and absolutely continuous with respect to Lebesgue), conservativity of $T_J$ is evident, thus by this and the results from \cite[Proposition 4.3.3]{Aaronson97}, we may apply \cite[Theorem 4.4.7]{Aaronson97} to return that $T_J$ is exact with respect to the induced measure $\mu_J$. Then, by \cite[Proposition 1.5.2]{Aaronson97}, $T_4$ is ergodic.

Adapting this method to prove ergodicity of the other systems is a matter of choosing an appropriate $J$ to induce on, though a good rule of thumb is to start with $J$ being a union of the length 1 cylinder sets where the invariant density is bounded, and add in length 2 cylinder sets where necessary to ensure that the pre-images of $J$ cover $X$.

\begin{exmp}
    We finish our through-line example of studying $T_5$ here. From Example \ref{exmplInvariantDensity}, we have that an invariant density for $T_5$ is $\frac{1}{t}\, \mathrm{d}t$, and we have the labelling $\phi_1, \ldots, \phi_5$ for the branches. Choose to induce on $J := \bigcup_{i=2}^5 [i] = \left( \frac{\sqrt{3}}{\sqrt{3}+1}, \sqrt{3} \right)$ so that the induced measure is finite. Showing that $T_5$ is conservative is reasonably simple with the same technique as the above example, so is omitted here. Thus, we only need to show that $(T_5)_J $ is a conservative expanding $C^2$ Markov interval map to prove ergodicity of $T_5$. 

    Since the corresponding induced measure on $J$ is finite and $T_J$-invariant, $T_J$ is automatically conservative. The partition 
    \[ \mathcal{P} = \left\{ \bigcup_{i=2}^5 [j 1^k i] : j \in \{1, \ldots, 5\}, k \geq 0 \right\}\]
    is easily verified to be a Markov partition, and on each partition element, $(T_5)_J$ is of exactly the form $\phi_1^k\circ \phi_j$ and maps the partition element to the entirety of $J$. The form of $(T_5)_J$ on $\bigcup_{i=2}^5 [2 1^k i]$ will be $\phi^n_1 \circ \phi_2$, which calculated explicitly is
    \begin{equation*} 
    f = \phi^n_1 \circ \phi_2 (t) = \frac{3-3t}{(3n+1)t - 3n}.
    \end{equation*}
    Call this $f$ to simplify the following. The end points of the set $\bigcup_{i=2}^5 [2 1^k i]$ are those points $t$ where $f(t) = \sqrt{3}$ and $f(t) = \frac{\sqrt{3}}{1+\sqrt{3}}$, which (when explicitly computed) yields 
    \begin{equation}\label{intervalExplicit} \bigcup_{i=2}^5 [2 1^k i] = \left( \frac{3n+\sqrt{3}+3}{3n+\sqrt{3}+4}, \frac{3n+\sqrt{3}}{3n + \sqrt{3} + 1} \right).\end{equation}
    The first and second derivatives of $f$ are
    \[ \frac{\mathrm{d}\,f}{\mathrm{d}t} = \frac{-3}{((3n+1)t - 3n)^2}, \quad \frac{\mathrm{d}^2f}{\mathrm{d}t^2} = \frac{6(3n+1)}{((3n+1)t - 3n)^3},\]
    and the values of $t$ for which $\left|\frac{\mathrm{d}\,f}{\mathrm{d}t}\right| > 1$ are exactly the interval $\left( \frac{6n - \sqrt{3}}{6n + 2}, \frac{6n + \sqrt{3}}{6n+2}\right)$ which entirely contains \eqref{intervalExplicit}. By evaluating the endpoints of \eqref{intervalExplicit} into the formula for $\frac{\mathrm{d}\,f}{\mathrm{d}t}$, we can also verify that as $n \rightarrow \infty$, $\frac{\mathrm{d}\,f}{\mathrm{d}t}$ remains bounded away from 1, whence we have that $\inf_{[2]} (T_5)_J > 1$. For the distortion property, note that 
    \[ \frac{\left|\frac{\mathrm{d}^2f}{\mathrm{d}t^2} \right|}{(\frac{\mathrm{d}\,f}{\mathrm{d}t})^2} = \frac{6}{9} (3n+1)((3n+1)t - 3n),\]
    which can also be confirmed to be finite over the interval \eqref{intervalExplicit}, and as $n \rightarrow \infty$, it will remain bounded and will not grow arbitrarily large. 

    The same argument applies for all other forms for the branches of $(T_5)_J$, hence we have that it is an expanding $C^2$ Markov interval map. By \cite[Proposition 4.3.3]{Aaronson97} and \cite[4.4.7]{Aaronson97}, $(T_5)_J$ is exact, and therefore ergodic, and consequently $T_5$ is ergodic (as $J$ can be confirmed to be a sweep-out set).
\end{exmp}

\section{The Farey Map as a Projection}
One further example that \cite{Cha2018} notes is for the quadratic form $Q(\underline{x}) = -x_1x_2 - x_2x_3 - x_3x_1$. The collection of matrices that generate solutions in $\{ x_1, x_2> 0,\ x_3<0 \}$ is 
\begin{equation}
    N_1 = \begin{bmatrix}
        2 & 0 & 1 \\
        2 & 1 & 0 \\
        -1 & 0 & 0 \\
    \end{bmatrix},\quad  
    N_2 = \begin{bmatrix}
        2 & 1 & 0 \\
        2 & 0 & 1 \\
        -1 & 0 & 0\\
    \end{bmatrix},
\end{equation}
with base triple $(2,2,-1)^T$. Choosing to normalise in the $x_1$-coordinate, the actions obtained operate on the arc of the ellipse $\{ x + y + xy = 0 \}$ defined by $x>0, y<0$, which will be denoted by $\mathcal{E}$. The collection of $\widehat{N_1^{-1}}, \widehat{N_2^{-1}}$ defines an endomorphism on this arc as usual. 

Choose the slightly different projection that takes any point $(x,y)\in \mathcal{E}$ to the gradient of the straight line joining $(x,y)$ and $(-2,-2)$ -- with the convention that $(\infty, -2)$ is mapped to $0$ -- the projection is explicitly written as
\begin{equation}
    \pi(x,y) = 2 \cdot \frac{y-x}{x+2}, \quad \pi^{-1}(t) = \left( \frac{1-t}{t}, \frac{(1-t)^2}{t} \right).
\end{equation}
The result of conjugating the map that is given piecewise by the $\widehat{N_i^{-1}}$ (on non-overlapping domains) under this projection yields exactly the Farey map:
\begin{equation*}
    \mathcal{F}(t) = \begin{cases}
        \frac{t}{1-t} & t \in \left(0, \frac{1}{2} \right), \\
        \frac{1-t}{t} & t \in \left( \frac{1}{2}, 1 \right).
    \end{cases}
\end{equation*}
Similarly to \cite{BANDI2023108871, calta24, PANTI202250}, the methods of computing extensions and the invariant measure of the Farey map still apply here. 

This raises an interesting question - which piecewise M\"{o}bius transformations can be exhibited as a projection of some solution generating system? Since many of these piecewise M\"{o}bius transformations are closely related to continued fraction style expansions, it is also reasonable to ask which of these continued fraction style maps are projections of these matrix systems via the method of Romik.

In a slightly different vein, Cha et al.\ make note that their methodology of producing these Berggren trees of matrices can generalise to higher dimensions. Garrity and Duke \cite{garrity2024} discuss ergodic properties of the slow and fast triangle maps, noting that these are one generalisation of the Farey map and continued fraction map respectively to higher dimensions that is particularly appropriate. Is it possible to exhibit these triangle maps in higher dimensions as projections of solution generating systems for higher dimensional quadratic forms?

\section{A Note on Pythagorean Quadruples}\label{QuadruplesSection}
As a final remark in \cite{Cha2018}, Cha et al.\ note that their method of constructing matrix collections that generate solutions is not necessarily restricted to quadratic forms in 3 variables. They compute a collection of seven matrices that generates all positive primitive Pythagorean quadruples, with the only major difference being that the associated Berggren tree has infinitely many root nodes. This can also be adapted into a dynamical system via constructing an algorithm for computing quadruple codes, albeit this system will now be 2-Dimensional. 

A natural following question is whether similar methods can be used to determine extensions of these systems, in order to compute invariant measures. However, there are some difficulties in this higher dimensional generalisation attempt. First, we would need the appropriate notion of discriminant in $n$ variables and show that this is invariant under the extension. Additionally, the approach requires a projection of this extension onto a map on $\mathbb{R}^2$, and the most intuitive way of defining the extension map is as an action on the space of quadratic polynomials in $n$ variables, which in higher dimensions than 1 is much more than $n+1$.

\section{Further Investigation}
While this project has reach a reasonable conclusion, there are still a few avenues of research that are not fully complete. Some of these have already been mentioned during the paper, though they will be explicitly stated again here. 

Firstly, the adaptation of Keane's method to lift the solution generating system to a Lebesgue preserving bijection on the space of quadratic polynomials in one variable is highly dependent on the condition \eqref{markovMapCondition} -- if we do not have such a domain $R$, then we cannot ensure bijectivity. However, regarding the construction of the matrices that generate the integer solutions, we can see that a region on the ellipse $\mathcal{E}$ exhibits similar properties. Hence, one may suspect that if we project this domain onto the real line using the appropriate modified stereographic projection map, we return the appropriate region $R$. I suspect that this is not difficult to verify, though requires some explicit computation to do which has not been done as of yet. Most notably, it is necessary to confirm that the conjugate dynamical system determined by the stereographic projection conjugacy map is indeed piecewise M\"{o}bius. If this is true, then not every piecewise M\"{o}bius map is a projection of some solution generating system (regard that the doubling map cannot have a region $R$ satisfying \eqref{markovMapCondition}), so it is also pertinent to ask what class of piecewise M\"{o}bius maps are achievable by this projection method.

As mentioned previously, one may also ask whether the extension system restricted to the parameter space with determinant 1 is indeed the natural extension of the base dynamical system. We do not see any reason why this would not be the case (given that the corresponding step in \cite{Keane95} may be viewed as the natural extension of the Gauss map), though this will likely take some calculation to verify. 

Given that the Farey map can be represented as a projection of a solution generating system for a particular quadratic form, it is worth asking whether other number-theoretic maps/objects have similar representations. Can modified Gauss maps be represented in a similar way (including $\alpha$-continued fraction maps)? Additionally, can we define continued fraction maps in higher dimensions by passing to an appropriate quadratic form in one more variable and then performing the method outlined in the section above to determine a map on $[0,1)^2$? 

One problem with the method is that it is difficult to guarantee whether a given quadratic form has a solution generating system or not. It may be interesting to find some family of quadratic forms that all have solution generating systems, and calculate the corresponding family of maps, determining whether they are ergodic, and potentially investigating the number-theoretic properties of the numeration systems they induce.

\section{Acknowledgements}
The author would like to thank several people for their aid, help and insight with the above work.

Firstly, a thanks to Tom Kempton for his supervision of the project as part of the PhD, and encouragement to carry it to its completion. It has been a long time in the works, and it has been a pleasure being your student.

Next, a thanks is also extended to Giovanni Panti, who provided the very useful insight of the De Sitter space, which provides more concrete connections between this work and others, and gives a possible framework for aiding in the proof of the proposed conjecture.

Finally, the author would like to thank Karma Dajani and Cor Kraaikamp for their friendly advice in the final stages of the project, in particular with regards to publishing advice.

\bibliographystyle{plain}
\bibliography{refs}

\end{document}